\documentclass[12pt]{article}
\usepackage[left=1in,right=1in,top=1in,bottom=1in,letterpaper]{geometry}

\def\x{{\mathbf x}}

\def\A{{\mathcal A}}
\def\C{{\mathcal C}}
\def\B{{\mathcal B}}
\def\D{{\mathcal D}}
\def\J{{\mathcal J}}

\def\X{{\mathcal X}}
\def\Z{{\mathcal Z}}
\def\U{{\mathcal U}}
\def\V{{\mathcal V}}
\def\S{{\mathcal S}}
\def\Y{{\mathcal Y}}

\def\W{{\mathcal W}}
\newcommand{\LSP}{\,\mathrm{LSP}}
\usepackage{amsopn}
\newcommand{\norm}[1]{\left\lVert#1\right\rVert}
\newcommand{\R}{\mathbb{R}}

\DeclareMathOperator*{\sign}{sign}
\DeclareMathOperator{\prox}{prox}

\DeclareMathOperator*{\argmin}{argmin}

\usepackage{amsfonts}
\usepackage{graphicx}
\usepackage{epstopdf}
\usepackage{algpseudocode}
\usepackage{algorithm}
\usepackage{multirow}
\usepackage{amsthm}

 \newtheorem{theorem}{Theorem}[section]
 \newtheorem{lemma}[theorem]{Lemma}
 \newtheorem{definition}[theorem]{Definition}
 \newtheorem{proposition}[theorem]{Proposition}
 \newtheorem{corollary}[theorem]{Corollary}

 \newtheorem{remark}[theorem]{Remark}

\usepackage{graphicx}        % standard LaTeX graphics tool
\usepackage{multicol}        % used for the two-column index

\usepackage{newtxtext}       %
\usepackage[varvw]{newtxmath}       % selects Times Roman as basic font

\begin{document}

\title{Log-Sum Regularized Kaczmarz Algorithms for High-Order Tensor Recovery}
\author{Katherine Henneberger\thanks{Department of Mathematics, University of Kentucky, Lexington, KY 40511}
\and Jing Qin\thanks{Department of Mathematics, University of Kentucky, Lexington, KY 40511}}
\date{}

\maketitle

\abstract{Sparse and low rank tensor recovery has emerged as a significant area of research with applications in many fields such as computer vision. However, minimizing the $\ell_0$-norm of a vector or the rank of a matrix is NP-hard. Instead, their convex relaxed versions are typically adopted in practice due to the computational efficiency, e.g., log-sum penalty. In this work, we propose novel log-sum regularized Kaczmarz algorithms for recovering high-order tensors with either sparse or low-rank structures. We present block variants along with convergence analysis of the proposed algorithms. Numerical experiments on synthetic and real-world data sets demonstrate the effectiveness of the proposed methods.}

\section{Introduction}
\label{sec:1}

The Kaczmarz algorithm \cite{karczmarz1937angenaherte} is one of the prominent algorithms for solving large linear systems, especially when the desired solution is sparse. It aims to solve a consistent linear system in an iterative manner by involving only one equation at each iteration. The Kaczmarz algorithm was later rediscovered as Algebraic Reconstruction Technique (ART)  in computational tomography \cite{gordon1970algebraic}. The efficiency of the Kaczmarz algorithm makes it a popular tool in many applications such as medical imaging \cite{zhou2013tensor} and signal recovery \cite{chen2021regularized}.

There exist numerous extensions of the classical Kaczmarz algorithm including randomized, regularized, and block interpretations. The randomized Kaczmarz algorithm (RKA) \cite{strohmer2009randomized}
employs a random selection of projections at each update, which can be considered as a special case of coordinate descent or stochastic gradient descent. If the rows of matrix $A$ are chosen randomly with uniform i.i.d. selection, then the RKA is an instance of stochastic gradient descent \cite{needell2014stochastic}. Additionally, when solving the least-squares problem, the RKA can be obtained by applying coordinate descent to the dual formulation of the objective function \cite{wright2015coordinate}. Overall, the RKA outperforms the classical Kaczmarz algorithm for overdetermined systems. In order to further improve convergence, regularized Kaczmarz algorithm introduces a regularization term into the objective function \cite{chen2021regularized}. Block Kaczmarz algorithms \cite{needell2014paved,needell2015randomized} are developed to seek acceleration in solving larger linear systems by performing each iteration on a block of indices selected from an independent partition of the row indices. More recently, optimal learning rate is discussed for the matrix RKA \cite{marshall2023optimal}.

By considering general multi-linear systems, the Kaczmarz algorithm and its variants have been extended to the tensor setting recently. In \cite{wang2023solving}, the classical matrix Kaczmarz algorithm is extended by considering a tensor-vector multiplication problem. The structure of tensor-vector multiplication allows for the consideration of high-order tensors, while Kaczmarz-type algorithms for tensor systems are more commonly applied to problems involving third-order tensor-tensor products. For example, \cite{ma2022randomized} introduces a randomized Kaczmarz method for solving linear systems of third-order tensors under the t-product. Notably, the tensor randomized Kaczmarz method is empirically shown to converge faster than the traditional Kaczmarz method applied to a matricization of the data. Furthermore, \cite{chen2021regularized} adapts the Kaczmarz algorithm into the sparse/low-rank third-order tensor recovery problem by proposing a regularized Kaczmarz algorithm. This work brings a regularization term into the objective function to preserve desirable characteristics of the underlying solution, such as sparsity and low-rankness. In our work, we will consider a general sparse or low-rank tensor recovery problem with tensor order at least three and propose log-sum based tensor regularizations.

Due to the good approximation to the $\ell_0$ pseudo-norm, the log-sum function has been widely applied as a regularizer in sparse and low-rank optimization problems. Recent literature has explored the utility of this function in compressed sensing \cite{prater2022proximity, zhou2023iterative,chen2021logarithmic}. In \cite{zhou2023iterative} an iterative thresholding algorithm is proposed for solving the sparse matrix recovery problem with a matrix log-sum regularization. While nuclear norm minimization is a popular model for matrix  completion, a log-sum penalty term composed with the singular values of the matrix has shown superior performance as a nonconvex rank surrogate in low-rank minimization \cite{chen2021logarithmic}.

To the best of our knowledge, most of the Kaczmarz-type tensor algorithms in literature focus on the third-order tensor case. The extension to tensors with order greater than three is not trivial as the classical t-product and the associated algebraic structures are only defined for third-order tensors. High-order tensors commonly arise in real-world applications such as in fourth-order color videos and hyperspectral data. Furthermore, there are a variety of tensor recovery applications, such as image inpainting, sparse signal recovery, and image destriping. Working directly with tensors and avoiding matricization is advantageous as it preserves the multidimensional structure and correlation of the data. Another benefit is the enhanced computational efficiency in processing the high-order tensorial data in parallel with utilized data structures. Motivated by the utility of log-sum regularization and the aforementioned Kaczmarz extensions, we propose a log-sum regularized Kaczmarz algorithmic framework for high-order tensor recovery with corresponding convergence analysis for several variants. We conduct multiple numerical experiments to demonstrate convergence and effectiveness of the proposed algorithms.

The major contributions of this paper are shown as follows:
\begin{enumerate}
    \item We propose two novel log-sum regularized Kaczmarz algorithms for solving the general high-order tensor recovery problem where the recovered tensor is either sparse or low-rank, and analyze the convergence of the proposed algorithms.
    \item We present two extensions to the sparse and the low-rank cases with corresponding convergence guarantees for block variants.
    \item Numerical experiments on synthetic and real world data justify the performance of the proposed methods. Furthermore, applications to image destriping demonstrate the efficiency of the algorithms in practice.
\end{enumerate}
The rest of the paper is organized as follows: Section~\ref{sec:prelim} introduces basic concepts and results for tensors and convex optimization. Our algorithms are proposed in Section~\ref{sec:Prop} along with the relevant convergence analysis. Block extensions of our algorithms are presented in Section~\ref{sec:Ext}, and various numerical experiments are described in Section~\ref{sec:num} to demonstrate the proposed effectiveness. Finally, the conclusion and future works are summarized in Section~\ref{sec:Con}.

\section{Preliminaries}
\label{sec:prelim}
Throughout the paper, we use bold lowercase letters such as $\x$ to denote vectors, capital letters such as $X$ to denote matrices, and calligraphic letters such as $\X$ to denote tensors. We denote the zero tensor as $\mathcal{O}$ and we use the bracket notation $[n]$ to denote the set of the first $n$ natural numbers, i.e., $\{1,2,...,n\}$. We use $\mathbb{E}$ to denote the expectation and we assume that the tensor $\X\in\R^{n_1\times n_2\times ...\times n_m}$ unless otherwise specified. To make this paper self-contained, we include necessary definitions for tensors and relevant operations from \cite{qin2022low} in Definition~\ref{def:tensors}.

\begin{definition}\cite{qin2022low}\label{def:tensors}

Consider the $m$-th order tensors $\A\in\R^{n_1\times n_2 \times n_3\times ...\times n_m}$ and $\B\in\R^{n_2\times k\times n_3\times ...\times n_m}$.
\begin{itemize}
      \item The \textbf{block diagonal matrix} of the tensor $\A$ is formed using the matrix faces $A^j=\A(:,:,i_3,...,i_m)$ where $j = (i_m-1)n_3\cdots n_{m-1}+\cdots+(i_4-1)n_3+i_3$, such that $\text{bdiag}(\A)=\text{diag}(A^1,A^2,...,A^J)\in\R^{n_1n_3\cdots n_m\times n_2n_3\cdots n_m}$ where $J=n_3n_4\cdots n_m$.
        \item The \textbf{facewise product} of $\A$ and $\B$, denoted by $\C=\A\Delta\B$, is defined through
         \begin{align*}
             \text{bdiag}(\C)=\text{bdiag}(\A)\cdot\text{bdiag}(\B),
         \end{align*}
         where $\C$ is of size ${n_1\times k\times n_3\times...\times n_m}$, and $(\cdot)$ is simply matrix multiplication.

        \item  { The \textbf{mode-$i$ product} of tensor $\A$ with matrix $U$, denoted as $\B = \A\times_i U$ is defined as $\B_{(i)} = U\cdot \A_{(i)}$ where $\B_{(i)}$ and $\A_{(i)}$ denote the mode-$i$ unfoldings of tensors $\B$ and $\A$ respectively and $(\cdot)$ denotes matrix multiplication.}
        \item A general \textbf{linear transform}
        $L:\mathbb{C}^{n_1\times \ldots\times n_m}\rightarrow \mathbb{C}^{n_1\times \ldots\times n_m}$ is defined as $$ L(\A) = \A\times_3 U_{n_3}\times ...\times_m U_{n_m}:=\A_L,$$ where $U_{n_i}\in{\mathbb{C}^{n_i\times n_i}}$ are the corresponding transform matrices. When matrices $\{U_{n_i}\}_{i=3}^m$ are invertible, the linear transform $L$ is invertible. Moreover, we assume there exists a constant $\rho>0$ such that the corresponding matrices $\{U_{n_i}\}_{i=3}^m$ of the invertible linear transforms $L$ satisfy the following condition:
        \[
        \begin{aligned}
            &(U_{n_m}\otimes U_{n_{m-1}}\otimes ...\otimes U_{n_3})\cdot(U^*_{n_m}\otimes U^*_{n_{m-1}}\otimes ...\otimes U^*_{n_3})\\ &=(U^*_{n_m}\otimes U^*_{n_{m-1}}\otimes ...\otimes U^*_{n_3}) \cdot(U_{n_m}\otimes U_{n_{m-1}}\otimes ...\otimes U_{n_3})\\
        &=\rho I,
        \end{aligned}\]
        where $U^*_{n_i}$ is the complex conjugate transpose of $U_{n_i}$ and $\otimes$ denotes the Kronecker product.
        Here the identity matrix is of size $n_3 n_4....n_m \times n_3 n_4....n_m $.
        For example, if $U_{n_i}$ for $i=3,\ldots,m$ are the matrix of the Fourier transform, corresponding to the command \verb"fft" in Matlab, then $\rho = n_3\cdot n_4\cdot....n_m$. The corresponding $L$ for this case is called Fourier transform. In general, if $U_{n_i}$'s are unitary, then $\rho=1$.
        \item The \textbf{L-based t-product} is defined as $*_L: \R^{n_1\times n_2\times ....\times n_m}\times \R^{n_2\times k\times ....\times n_m}\rightarrow \R^{n_1\times k\times ....\times n_m}$ such that
        \begin{equation}
            \A*_L\B = L^{-1}(\A_L\Delta \B_L)=\C,
            \label{ltprod}
        \end{equation}
        where $\Delta$ is the facewise product. When $L$ is the Fourier transform, this reduces to the standard t-product. This yields two important properties:
\begin{align}
    \langle \A,\B\rangle  &= \frac{1}{\rho}\langle \text{bdiag}(\A_L),\text{bdiag}(\B_L)\rangle \label{property1},\\
      \|\A\|^2_F&=\frac{1}{\rho}\|\text{bdiag}(\A_L)\|_F^2.\label{froprop}
\end{align}
\item The \textbf{conjugate transpose} of $\A\in \mathbb{C}^{n_1\times n_2\times \ldots\times n_m}$ is the tensor $\A^*\in \mathbb{C}^{n_2\times n_1\times \ldots\times n_m}$ such that
\begin{equation}
(\A^*)_L(:,:,i_3,...,i_m) = (\A_L(:,:,i_3,...,i_m))^*\label{property2}
\end{equation}
for all $i_j\in[n_j]$, $j\in{3,...,m}$.

\item Tensor $\J\in \R^{n\times n\times n_3 ...\times n_m}$ is the \textbf{identity tensor} if $\J_L(:,:,i_3,...,i_m)=I_n$ for all $i_j\in [n_j], j\in\{3,...,m\}$.

\item
Tensor $\mathcal{X}$ is \textbf{orthogonal }if $\mathcal{X}^**_L\mathcal{X}=\mathcal{X}*_L\mathcal{X}^*=\J$.

\item Tensor $\X$ is \textbf{f-diagonal} if each frontal slice $\X(:,:,i_3,...,i_m)$ is diagonal for all  $i_j\in [n_j], j\in\{3,...,m\}$.

\item  The \textbf{tensor Singular Value Decomposition (t-SVD)} of tensor $\X\in\R^{n_1\times n_2\times n_3\times\ldots\times n_m}$ induced by $*_L$ is
\[
\mathcal{X}=\mathcal{U}*_L\mathcal{S}*_L\mathcal{V}^*\] where $\mathcal{U}\in\R^{n_1\times n_1\times n_3\times\ldots\times n_m}$ and $\mathcal{V}\in\R^{n_2\times n_2\times n_3\times \ldots\times n_m}$ are orthogonal and the core tensor $\mathcal{S}\in\R^{n_1\times n_2\times n_3\times\ldots\times n_m}$ is f-diagonal.
   \end{itemize}
\end{definition}

Since calculating the t-product with the Fourier transform may require complex arithmetic, the general linear transform $L$ is employed to extend the classical definition of the t-product. With this definition we are not limited to using the Fourier transform to compute the t-product, and instead may employ any invertible linear transform such as the discrete cosine transform (DCT) when each $U_{n_i}$ is a DCT transform matrix.

\begin{corollary}\label{cor:sep}
There exist several separability properties about the generalized t-product including separability in the first dimension. Consider $\A\in\R^{n_1\times n_2 \times n_3\times ...\times n_m}$ and $\X\in\R^{n_2\times k\times n_3\times ...\times n_m}$, and let $\A(i) = \A(i,:,:,;...,:)$ denote the $i$th horizontal slice of size $1\times n_2\times n_3...\times n_m$. Then,
\[ (\A *_L \X) (i) = L^{-1}(\A_L(i)\Delta \X_L).\]
\end{corollary}

\begin{remark}
    Note, in \cite{qin2022low} and \cite{lu2018tensor}, the algorithm used to calculate the conjugate transpose acts on the frontal slices of $\A$. Thus, for any tensor $\A$, we have
\begin{align} \label{bdiag}
    (\text{bdiag}(\A))^* = \text{bdiag}(\A^*).
\end{align}
By \eqref{property2} and the fact that the conjugate transpose can be implemented on the frontal slices, we have
\begin{align}\label{property3}
    (\A_L)^*(:,:,i_3,...,i_m) = (\A_L(:,:,i_3,...,i_m))^* = (\A^*)_L(:,:,i_3,...,i_m).
\end{align}
\end{remark}

\begin{lemma}\label{conjugate} If $\A\in\R^{n_1\times n_2 \times n_3\times ...\times n_m}$, $X\in\R^{n_2\times k\times n_3\times ...\times n_m}$ and $B\in\R^{n_1\times k\times n_3\times ...\times n_m}$, then
 \[\langle \A*_L\X,\B\rangle = \langle \X, (\A^*)*_L\B\rangle.\]
 \end{lemma}
\begin{proof}

It is straightforward to verify that $L(\A*_{L}\X)=\A_L\Delta \X_L$. Next we compute the inner product of a t-product and a tensor as follows
\begin{align*}
    \langle \A*_L\X,\B\rangle &=\langle L^{-1}(\A_L\Delta\X_L),\B\rangle \\
    &= \frac{1}{\rho}\langle \text{bdiag}(\A_L\Delta\X_L),\text{bdiag}(\B_L)\rangle& \text{by
 }\eqref{property1}\\
    &= \frac{1}{\rho}\langle \text{bdiag}(\A_L)\text{bdiag}(\X_L),\text{bdiag}(\B_L)\rangle\\
    &= \frac{1}{\rho}\langle \text{bdiag}(\X_L),(\text{bdiag}(\A_L))^*\text{bdiag}(\B_L)\rangle\\
    &= \frac{1}{\rho}\langle \text{bdiag}(\X_L),\text{bdiag}((\A_L)^*)\text{bdiag}(\B_L)\rangle
    & \text{by }\eqref{bdiag}\\
    &= \frac{1}{\rho}\langle \text{bdiag}(\X_L),\text{bdiag}((\A_L)^*\Delta \B_L)\rangle \\
    &= \frac{1}{\rho}\langle \text{bdiag}(\X_L),\text{bdiag}((\A^*)_L\Delta \B_L)\rangle&\text{by
 }\eqref{property3}\\
    &= \langle \X,\A^**_L\B)\rangle& \text{by
 }\eqref{property1}
\end{align*}
\end{proof}

\begin{definition}
    A function $f: \R^{n_1\times n_2\times ...\times n_m}\rightarrow \R$ is \textbf{$\alpha$-strongly convex }for some $\alpha>0$ if
 it satisfies
 \begin{align}f(\Y)\geq f(\X)+\langle\X^*,\Y-\X\rangle +\frac{\alpha}{2}\|\Y-\X\|_F^2\label{alphaconvex}\end{align}
 for any $\X,\Y\in \R^{n_1\times n_2\times ...\times n_m}$, $\X^*\in\partial f(\X)$. Here $\partial f(\X)$ denotes the \textbf{subdifferential} of $f$ at $\X\in\R^{n_1\times n_2\times ...\times n_m}$ and is defined as
\[\partial f(\X) = \{\X^* : f(\Y)\geq f(\X) + \langle \X^*,\Y-\X\rangle \text{   for any  } \Y\in\R^{n_1\times n_2\times ...\times n_m}\}.\]
\end{definition}
In particular, if $f:\R^n\to\R$ is twice continuously differentiable, then $f$ is $\alpha$-strongly convex if and only if the Hessian matrix $\nabla^2 f\succeq \alpha I_{n}$, i.e., $\nabla^2f-\alpha I_{n}$ is positive semi-definite.

\begin{definition}  A function $f:\R^{n_1\times n_2\times ...\times n_m}\rightarrow \R$ is \textbf{admissible} if the dual problem $g_f$ as defined in \cite[Section 3.2]{chen2021regularized} is restricted strongly convex on any of $g_f$'s levels sets. Function $f$ is \textit{strongly admissible} if $g_f$ is restricted strongly convex on $\R^{n_1\times k\times n_3\times ...\times n_m}$.
\end{definition}

\section{Proposed Algorithms}\label{sec:Prop}
In this section, we propose two novel algorithms based on the regularized Kaczmarz algorithm, i.e., one for sparse tensor recovery and one for low-rank tensor recovery. In both algorithms, we consider a log-sum penalty (LSP) regularization which is commonly used as a replacement for the $l_0$ pseudo-norm in compressive sensing and low-rank matrix recovery.

\subsection{Sparse tensor recovery}
Given two tensors $\A\in\R^{n_1\times n_2\times n_3\times ...\times n_m}$ and $\B\in\R^{n_1\times l\times n_3\times ...\times n_m}$, we assume the underlying sparse tensor $\X\in\R^{n_2\times l\times n_3\times ...\times n_m}$ satisfies the linear relationship $\A*_L\X=\B$ where $*_L$ denotes the generalized t-product defined in \eqref{ltprod}. To recover $\X$ that is assumed to be sparse, we consider the following linear constrained tensor minimization problem:
\begin{equation}\hat{\X}=
\argmin_{\X}\lambda \LSP(\X)+\frac12\norm{\X}_F^2\quad\mbox{s.t.}\quad
\A*_{L}\X=\B,
\label{LSPmodel}
\end{equation}
where the regularization parameter $\lambda>0$, the tensor LSP regularizer is defined as
\[
\LSP(\X)=\sum_{i_1=1}^{n_1}\cdots\sum_{i_m=1}^{n_m}\log\left(1+\frac{|\X_{i_1,...,i_m}|}{\varepsilon}\right)
\]
and the parameter $\varepsilon>0$ controls the slope of the cusp at the origin. Here the tensor LSP regularizer is employed to promote the sparsity of the estimated tensor.

According to Corollary \ref{cor:sep}, the linear constraint $\A*_L\X = \B$ can be split into $n_1$ constraints
\begin{align}
    \A(i)*_L\X = \B(i),\quad i\in[n_1].
\end{align}
We define the row space as, $R(\A):=\{\A^T*_L\Y:\Y\in\R^{n_1\times l\times n_3\times ...\times n_m}\}$
and we assume the optimal solution $\hat{\X}$ to the problem \eqref{LSPmodel} satisfies
\[ \partial f(\hat{\X})\cap R(\A)\neq \emptyset.
\]
where $f(\hat{\X}) = \lambda \LSP(\hat{\X})+\frac12\norm{\hat{\X}}_F^2$. This assumption implies at least one subgradient of $f(\hat{\X})$ can be represented by $\A$ using the current tensor product $*_L$.
Since the objective function in \eqref{LSPmodel} is strongly convex under certain conditions (see Lemma~\ref{lem:strcvx}) and the constraint is linear, we develop an efficient algorithm based on the regularized Kaczmarz algorithm \cite{chen2021regularized}, which shows promising efficiency especially when the data size is large. Then we obtain two subproblems, one of which involves the proximity operator of the LSP regularizer.

The proximity operator for the vector LSP has been explored in recent literature  \cite{prater2022proximity,gong2013general}, thus we can extend these results to develop a Kaczmarz algorithm for \eqref{LSPmodel}. Specifically, we define the proximity operator for tensor LSP as follows:
\begin{equation}\label{eqn:LSP}
\prox_{\lambda \LSP}(\Z)
=\argmin_{\X}\lambda\LSP(\X)+\frac12\norm{\X-\Z}_F^2.
\end{equation}
Since $\LSP(\X)$ is additively separable, $\prox_{\lambda \text{LSP}}(\X)$ can be calculated component-wise using a closed form detailed in the Appendix of \cite{gong2013general} and extended to tensors in Algorithm \ref{alg:lspprox} where we choose $\lambda$ and $\varepsilon$ such that \begin{align}
\left(|z|+\varepsilon\right)^2\geq 4\lambda\label{paramcriteria}
\end{align} for all elements $z$ in tensor $\Z$. By adapting the regularized Kaczmarz algorithm from \cite{chen2021regularized} to the current $*_L$-based tensor product and log-sum regularized objective function, we propose a novel and efficient algorithm, detailed in Algorithm~\ref{alg:LSP}, which solves \eqref{LSPmodel} and uses one horizontal slice of $\A$ at each iteration. The existence of a closed-form solution for solving the LSP proximity operator enhances the efficiency of this algorithm.

\begin{algorithm}
\caption{Computing $\prox_{\lambda \text{LSP}}(\Z)$ for a tensor $\Z$}\label{alg:lspprox}
\begin{algorithmic}
\State\textbf{Input:}  $\Z\in\R^{n_1\times n_2\times ...\times n_m}$, $\varepsilon>0$, and $\lambda>0$ chosen such that \eqref{paramcriteria} holds.
\State\textbf{Output:} $\prox_{\lambda \text{LSP}}(\Z)$
\For{ each element $z$ of $\Z$}
%\State Calculate:
\State $\Tilde{z}_1 = \max\left\{0,\frac{1}{2}\left[(|z|-\varepsilon)+\sqrt{(|z| + \varepsilon)^2 - 4 \lambda}\right]\right\}$
\State $\Tilde{z}_2 = \max\left\{0,\frac{1}{2}[(|z|-\varepsilon)-\sqrt{(|z| + \varepsilon)^2 - 4 \lambda}]\right\}$
%\State Calculate:
\State $\Tilde{z}= \argmin_{x\in\{0,\Tilde{z}_1,\Tilde{z}_2\}} \frac{1}{2}(x-|z|)^2+\lambda \log(1+ \frac{x}{\varepsilon})$
\State $\prox_{\lambda \text{LSP}}(z) = \sign(z) \Tilde{z}$
\EndFor
\end{algorithmic}
\end{algorithm}

\begin{algorithm}[H]
\caption{LSP regularized Kaczmarz for sparse tensor recovery (LSPK-S)}\label{alg:LSP}
\begin{algorithmic}
\State\textbf{Input:} $\B\in\R^{n_1\times l\times n_3\times ...\times n_m}$, $\A\in\R^{n_1\times n_2\times n_3\times...\times n_m}$, $\lambda, \varepsilon>0$ chosen such that \eqref{paramcriteria} holds, stepsize $t>0$, max number of iterations $T$, invertible linear transform $L$, and tolerance $tol$.
\State\textbf{Output:} An approximation of $\hat{\X}$
\State\textbf{Initialize:}  $\Z^{(0)}\in R(\A)\subset \R^{n_2\times l\times n_3\times ...\times n_m}, \X^{(0)} = \prox_{\lambda \text{LSP}}(\Z^{(0)})$
\For{ $k={0},...,T-1$ }
\State choose index $i(k)$ cyclically or randomly from $[n_1]$
\State $\Z^{(k+1)} = \Z^{(k)}+t\A^*(i(k))*_L\frac{\B(i(k))-(\A*_L\X^{(k)})(i(k))}{||\A(i(k))||_F^2}   $
\State $\X^{(k+1)} = \prox_{\lambda \text{LSP}}(\Z^{(k+1)}) $
\State{Terminate if} $||\X^{(k+1)}-\X^{(k)}||_F/||\X^{(k)}||_F<tol$
\EndFor
\end{algorithmic}
\end{algorithm}

\subsection{Low-rank tensor recovery}
Since low-rankness of a matrix can be reflected by the sparsity of its singular values from SVD, we can extend this idea to the tensor case using the previous LSP regularized Kaczmarz algorithm from sparse tensor recovery to low-rank tensor recovery by applying the LSP regularization to the middle f-diagonal tensor in the t-SVD form of a tensor.
\begin{definition}
    The Nuclear LSP (NLSP) norm is defined as
    \begin{align}\label{eqn:NLSP}
        \|\X\|_{\text{NLSP}} = \sum_{i=1}^{\min\{n_1,n_2\}} \text{LSP}(\S(i,i,:,\ldots,:)),
    \end{align}
    where $ \U*_L\S*_L\V^T$ is a t-SVD representation of $\X$.
\end{definition}

Equipped with the NLSP regularizer, we consider the following low-rank tensor recovery problem
\begin{align}
\min_{\X}\lambda  \|\X\|_{\text{NLSP}}+\frac12\norm{\X}_F^2\quad\mbox{s.t.}\quad
\A*_{L}\X=\B.
\label{LSPsingmodel}
\end{align}
An explicit form of the proximity operator for LSP composed with the singular value function for matrices is derived in \cite{prater2022proximity}. We extend this result to the tensor case to find a closed-form solution to \eqref{LSPsingmodel}. Denote $\mathcal{P}(\X) = \{(\U,\V): \X = \U*_L\S*_L\V^T\}$. That is, $\mathcal{P}(\X)$ is the set of all pairs $(\U,\V)$ such that $ \U*_L\S*_L\V^T$ is a t-SVD representation of $\X$.
\begin{definition}
    The tensor singular value proximity operator of the nuclear LSP at $\Z\in \R^{n_1\times n_2\times...\times n_m}$ is defined as
    \begin{align}
\prox_{\lambda\text{NLSP}}(\Z)=\argmin_{\X}\lambda  \|\X\|_{\text{NLSP}}+\frac12\norm{\X-\Z}_F^2.
\label{LSPsinggen}
\end{align}
\end{definition}

\begin{theorem}\label{thm:lspsingprox}
For each $\Z\in \R^{n_1\times n_2\times...\times n_m}$, if $\X^*$ is a solution to \eqref{LSPsinggen}, then there exists a pair $(\U,\V)\in\mathcal{P}(\mathcal{Z})$ such that $\Z = \U*_L \S *_L \V^*$ and a tensor $\D\in \prox_{\lambda \LSP}(\S)$ such that $\X^*= \U*_L\D*_L\V^T$.
\end{theorem}
\begin{remark} This is an extension of the proof in Theorem $4$ in \cite{prater2022proximity}. We show that $\X^*= \U*_L\D*_L\V^T$ is a solution to \eqref{LSPsinggen}. Plugging \eqref{eqn:NLSP} into \eqref{LSPsinggen} yields
 \begin{align}
\min_{\X} \sum_{i=1}^{\min\{n_1,n_2\}} \lambda \text{LSP}(\D(i,i,:,...,:))+\frac12\norm{\X-\Z}_F^2
\label{LSPsinggenref}
\end{align}
where $\D$ is the core tensor in the t-SVD representation of $\X$.
We can verify that
\begin{align}
    \|\X-\Z\|_F^2  =
     &\|\X\|_F^2-2\langle\X,\Z\rangle+\|\Z\|_F^2\nonumber\\
      \geq&\|\D\|_F^2-2\langle \D,\S\rangle+\|\S\|_F^2=\norm{\D-\S}_F^2\label{normsimp}
\end{align}
where the last inequality is due to von Neumann's trace inequality and $\S$ is the core tensor in the t-SVD representation of $\Z$.
Then the optimization problem \eqref{LSPsinggenref} reduces to
\begin{align}
\min_{\D} \sum_{i=1}^{\min\{n_1,n_2\}} \lambda \text{LSP}(\D(i,i,:,...,:))+\frac12\norm{\D(i,i,:,...,:)-\S(i,i,:,...,:)}_F^2.
\label{LSPsingsep}
\end{align}
This objective function is now separable and the solution to \eqref{LSPsingsep} is $\prox_{\lambda \text{LSP}}(\S)$ from above.
\end{remark} Similar to the sparse case above, we adapt the regularized Kaczmarz algorithm from \cite{chen2021regularized} to the current $*_L$-based tensor product and consider the regularization term as the log-sum function composed with the singular value function. Thus we propose a new algorithm for solving the low-rank tensor recovery problem \eqref{LSPsingmodel}, which is described in Algorithm~\ref{alg:LSPsing}.

\begin{algorithm}[ht]
\caption{LSP regularized Kaczmarz for low-rank tensor recovery (LSPK-L)}\label{alg:LSPsing}
\begin{algorithmic}
\State\textbf{Input:} $\B\in\R^{n_1\times l\times n_3\times ...\times n_m}$, $\A\in\R^{n_1\times n_2\times n_3\times...\times n_m}$, $\lambda,\varepsilon>0$ chosen such that \eqref{paramcriteria} holds, stepsize $t>0$, max number of iterations $T$, invertible linear transform $L$, and tolerance $tol$.
\State\textbf{Output:} An approximation of $\hat{\X}$
\State\textbf{Initialize:} $\Z^{(0)}\in R(\A)\subset \R^{n_2\times l\times n_3\times ...\times n_m}$ calculate $\text{t-SVD}(\Z^{(0)})=\U*_L\S*_L\V^*$, $ \X^{(0)} = \U*_L\prox_{\lambda \text{LSP}}(S)*_L\V^* $
\For{ $k=0,...,T-1$ }
\State choose index $i(k)$ cyclically or randomly from $[n_1]$
\State $\Z^{(k+1)} = \Z^{(k)}+t\A^*(i(k))*_L\frac{\B(i(k))-(\A*_L\X^{(k)})(i(k))}{||\A(i(k))||_F^2}   $
\State $\X^{(k+1)} = \prox_{\lambda \text{NLSP}}(\Z^{(k+1)})$
\State{Terminate if} $||\X^{(k+1)}-\X^{(k)}||_F/||\X^{(k)}||_F<tol$
\EndFor
\end{algorithmic}
\end{algorithm}

\subsection{Complexity analysis}

For Algorithm \ref{alg:lspprox}, the overall complexity for an input tensor $\Z\in\R^{n_1\times\ldots\times n_m}$ is $O(n_1 n_2\cdots n_m)$. For Algorithm \ref{alg:LSP}, the complexity in each iteration contains two parts, i.e., updating $\Z^{(k+1)}$ and updating $\X^{(k+1)}$. Updating $\Z^{(k+1)}$ is dominated by the t-product and thus when FFT is chosen as the transform $L$, the complexity is $O(n_2\cdots n_m\log(n_3\cdots n_m)+n_2ln_3\cdots n_m)$. Since $\X^{(k)}$ is of size $n_2\times l\times n_3\times ...\times n_m$ we know from the proximity operator complexity that updating $\X^{(k+1)}$ will be $O(n_2 l n_3\cdots n_m)$. Likewise, complexity of Algorithm \ref{alg:LSPsing} in each iteration consists of two parts. Updating $\Z^{(k+1)}$ has the same cost as that in Algorithm \ref{alg:LSP}. Updating $\X^{(k+1)}$ requires $O(n_2\cdots n_m\log(n_3\cdots n_m)+n_2ln_3\cdots n_m+n_2l^2n_3...n_m)$. Note that for a general linear transform $L$ with transform matrices $\{U_{n_i}\}_{i=3}^m$, the $*_L$-based t-product requires a higher computational complexity of $O(n_2(n_3\cdots n_m)^2+n_2ln_3\cdots n_m)$.

\subsection{Convergence analysis}

In this section, we provide the convergence guarantees for the proposed algorithms. As shown in \cite{chen2021regularized}, strong convexity is a sufficient condition for ensuring the convergence of regularized Kaczmarz algorithms for tensor recovery, whose extension to the $*_L$-based high-order t-product is straightforward. Following the convergence analysis pipeline of \cite{chen2021regularized}, we begin by showing the strong convexity of the objective function \eqref{LSPmodel}. Subsequently, we present our main convergence results in Theorems \ref{thm:cycconv} and \ref{thm:randconv} for Algorithm \ref{alg:LSP} with index $i(k)$ chosen cyclically and randomly respectively.

\begin{lemma}\label{lem:strcvx}
    The objective function \eqref{LSPmodel} is $\alpha$-strongly convex when $\sqrt{\lambda}\leq \varepsilon$.
\end{lemma}
\begin{proof}Since the objective function is additively separable, it is sufficient to examine the behavior of a single-variable function of the form
\begin{equation}
    h(x) = \frac{1}{2}|x|^2 + \lambda \log\left(1+\frac{|x|}{\varepsilon}\right)
\end{equation}
and bound the second derivative from below. Specifically, we have
Specifically, we have
\begin{align*}
    h''(x) = 1 - \frac{\lambda}{\varepsilon^2 (1+\frac{|x|}{\varepsilon})^2}\in[1-\lambda/\varepsilon^2,1),\quad x\neq 0.
\end{align*}
Thus, when $\sqrt{\lambda}<\varepsilon$,  $h(x)$ is $\alpha$-strongly convex with $\alpha=1 - \frac{\lambda}{\varepsilon^2}$ which implies that the inequality in \eqref{alphaconvex} holds. The objective function \eqref{LSPmodel} can be rewritten as
\begin{align}
    f(\X) = \sum_{i_1=1}^{n_1}\cdots \sum_{i_m=1}^{n_m} h(\X_{i_1,i_2,...,i_m}).
\end{align}
Thus, $f(\X)$ is $\alpha$-strongly convex with $\alpha=1 - \frac{\lambda}{\varepsilon^2}$.
\end{proof}
\begin{lemma} If $\A\in\R^{n_1\times n_2 \times n_3\times ...\times n_m}$ and $\X\in\R^{n_2\times l\times n_3\times ...\times n_m}$, then we have
 \begin{align}
     \|\A*_L\X\|_F\leq \sqrt{\rho}\|\A\|_F\|\X\|_F.\label{convlemma1}
 \end{align}
 \end{lemma}
 \begin{proof}
 By following the definition of the Frobenius norm for tensors, we have
\begin{align*}
    \|\A*_L\X\|_F^2 &= \|L^{-1}(\A_L\Delta\X_L)\|_F^2\\
    & = \frac{1}{\rho}\|\text{bdiag}(\A_L\Delta\X_L)\|_F^2 &\text{by \eqref{property1}}\\
    & = \frac{1}{\rho}\|\text{bdiag}(\A_L)\text{bdiag}(\X_L)\|_F^2\\
    & \leq \frac{1}{\rho}\|\text{bdiag}(\A_L)\|_F^2\|\text{bdiag}(\X_L)\|_F^2\\
    & =  \frac{1}{\rho}\left(\rho\|\A\|_F^2\right)\left(\rho\|\X\|_F^2\right) &\text{by \eqref{property1}}\\
    & = \rho \|\A\|_F^2\|\X\|_F^2.
\end{align*}
After taking the square root on both sides, we obtain the desired result. Note that this result is consistent with that for third-order tensors with Fourier transform.
\end{proof}
In what follows, we will define a more general version of Bregman distance, as well as its decay property for the proposed algorithms.

\begin{definition}\label{bregman}

For any convex function $f: \R^{n_1\times n_2\times ...\times n_m}\rightarrow \R$, the \textbf{Bregman distance} between $\X$ and $\Y$ with respect to $f$ and $\X^*\in \partial f(\X)$ is defined as
\[D_{f,\X^*}(\X,\Y) = f(\Y)-f(\X)-\langle \X^*,\Y-\X\rangle.\]
\end{definition}
\begin{definition} The \textbf{convex conjugate} of $f$ at $\Z$ is defined as
\[f^*(\Z) = \sup_{\X}\{\langle\Z,\X\rangle - f(\X)\}.\]
\end{definition}
Due to the duality property $\langle \X,\X^*\rangle = f(\X)+f^*(\X^*)$ when $\X^*\in \partial f(\X)$ \cite[Theorem 23.5]{rockafellar2015convex}, the Bregman distance can be written as
\[D_{f,\X^*}(\X,\Y) = f(\Y)+f^*(\X^*)-\langle \X^*,\Y\rangle.\] In general, if $f$ is $\alpha$-strongly convex, then the Bregman distance satisfies
\begin{align}\label{bregmanprop}
    \frac{\alpha}{2}\|\X-\Y\|_F^2\leq D_{f,\X^*}(\X,\Y).
\end{align}

\begin{proposition}\label{Prop1}
Suppose $\A\in\R^{1\times n_2 \times n_3\times ...\times n_m}$, $\B\in\R^{1\times l\times n_3\times ...\times n_m}$ and $f$ is an $\alpha$-strongly convex function defined on $\R^{n_2\times l\times n_3\times ...\times n_m}$. Given arbitrary $\overline{\Z}\in \R^{n_2\times l\times n_3\times ...\times n_m}$ and $\overline{\X}=\nabla f^*(\overline{\Z})$, let $\Z = \overline{\Z}+t\A^**_L\frac{\B-(\A*_L\overline{\X})}{||\A(i(k))||_F^2}$  and $\X = \nabla f^*(\Z).$ Then we have

\[D_{f,\Z}(\X,\mathcal{H})\leq D_{f,\overline{\Z}}(\overline{\X},\mathcal{H}) - \frac{t}{\|\A\|_F^2}\left(1-\frac{t \rho}{2\alpha}\right)\|\B-\A*_L\overline{\X}\|_F^2\]
for any $\mathcal{H}$ that satisfies $\A*_L\mathcal{H} = \B$.
\end{proposition}
\begin{proof}
    Let $\W = \A^**_L(\B-\A*_L\overline{\X})$ and $s = \frac{t}{\|\A\|_F^2}$. Then $\Z = \overline{Z}+s\W$ and by \eqref{convlemma1} we have
    \[\|\W\|_F = \|\A^**_L(\B-\A*_L\overline{\X})\|_F = \sqrt{\rho}\|\A\|_F\|\B-A*_L\overline{\X}\|_F.\]

By \eqref{property1} we have
\begin{align*}
    \langle \W,\mathcal{H}-\overline{\X}\rangle& = \langle  \A^**_L(\B-\A*_L\overline{\X}),\mathcal{H}-\overline{\X}\rangle\\
    &=  \langle  \A^**_L(\A*_L\mathcal{H}-\A*_L\overline{\X}),\mathcal{H}-\overline{\X}\rangle\\
    &=\langle  \A*_L\mathcal{H}-\A*_L\overline{\X}, \A*_L(\mathcal{H}-\overline{\X})\rangle\\
    &=\|\A*_L\mathcal{H}-\A*_L\overline{\X}\|_F^2\\
    &=\|\B-\A*_L\overline{\X}\|_F^2
\end{align*}
Then using these results the proof follows as in \cite{chen2021regularized}.
\end{proof}
Next, we will provide the convergence guarantee for Algorithm \ref{alg:LSP} when the index $i(k)$ is chosen cyclically.
\begin{theorem}\label{thm:cycconv}
Consider the objective function
\begin{align}
    f(\X) = \sum_{i_1=1}^{n_1}\cdots\sum_{i_m=1}^{n_m}\lambda\log\left(1+\frac{|\X_{i_1,...,i_m}|}{\varepsilon}\right)+\frac12\norm{\X}_F^2.\label{objfunc}
\end{align}
  If $t<\left(2 - \frac{2\lambda}{\varepsilon^2}\right)/\rho$, then the sequence $\{\X^{(k)}\}$ generated cyclically in Algorithm \ref{alg:LSP} satisfies
\begin{align}
    &D_{f,\Z^{(k+1)}}(\X^{(k+1)},\X)\notag\\
    &\leq D_{f,\Z^{(k)}}(\X^{(k)},\X) - t\left(1-\frac{t\rho}{\left(2 - \frac{2\lambda}{\varepsilon^2}\right)}\right)\frac{\| A(i(k))*_L(\X^{(k)}-\X)\|_F^2}{\|\A(i(k))\|_F^2}. \label{convthm}
\end{align}
Moreover, the sequence $\{\X^{(k)}\}$ converges to the solution of
\[\argmin_{\X}f(\X)\quad \mbox{s.t. }\quad{\A}*_L\X=\B.\]
\end{theorem}

\begin{remark}
The restriction $t<(2 - \frac{2\lambda}{\varepsilon^2})/\rho$ guarantees that $\sqrt{\lambda}<\varepsilon$, which in turn guarantees that the function $f$ is $\alpha$-strongly convex with $\alpha = 1 - \frac{\lambda}{\varepsilon^2}$. Thus, the proof of \eqref{convthm} follows directly from the proposition above when we take
    \[
        \A(i(k)) = \A,\quad
        \B(i(k)) = \B,\quad
        \Z^{(k)} = \overline{Z},\quad
        \Z^{(k+1)} = \Z.
    \] The proof of the convergence of $\{\X^{(k)}\}$ can be extended from the proof in \cite[Theorem 3.3]{chen2021regularized}.
\end{remark}

When the sequence of $i(k)$ is chosen randomly, it can be shown Algorithm 2 converges linearly in expectation.
To start with, we adapt \cite[Lemma 3.6]{chen2021regularized} to the current high-order tensor setting.
\begin{lemma}\label{lemma:nu}
 Assume $f$ is admissible, and let $\hat{\X}$ be the solution to \eqref{LSPmodel}. For $\Tilde{\X}$ and $\Tilde{\Z}\in \partial f(\Tilde{\X})\cap R(\A)$, there exists $\nu>0$ such that for all $\X$ and $\Z\in \partial f(\X)\cap R(\A)$ with $D_{f,\Z}(\X,\hat{\X})\leq D_{f,\Tilde{\Z}}(\Tilde{\X},\hat{\X})$, it holds that
\begin{align}\label{nulemma}
    D_{f,\Z}(\X,\hat{\X})\leq \frac{1}{\nu}\|\A*_L(\X-\hat{\X})\|_F^2.
\end{align}
\end{lemma}
Thus for Algorithm \ref{alg:LSP} with index $i(k)$ chosen randomly we have the following result.

\begin{theorem}\label{thm:randconv}
     Consider the objective function
\begin{align}
    f(\X) = \sum_{i_1=1}^{n_1}\cdots\sum_{i_m=1}^{n_m}\lambda\log\left(1+\frac{|\X_{i_1,...,i_m}|}{\varepsilon}\right)+\frac12\norm{\X}_F^2.
\end{align}
  If $\sqrt{\lambda}<\varepsilon$, and $0<\frac{\nu t}{\|\A\|_F^2}(1-\frac{t \rho\varepsilon^2}{(2\varepsilon^2-2\lambda)})<1$, then $\X^{(k)}$ converges linearly to the solution $\hat{\X}$ in expectation. Specifically,

\[\mathbb{E}\left[D_{f,\Z^{(k+1)}}(\X^{(k+1)},\hat{\X})\right]\leq \left(1-\frac{\nu t}{\|\A\|_F^2}(1-\frac{t \rho\varepsilon^2}{(2\varepsilon^2-2\lambda)})\right)\mathbb{E}\left[D_{f,\Z^{(k)}}(\X^{(k)},\hat{\X})\right].\]
Therefore:
\[\mathbb{E}\|\X^{(k)}-\hat{\X}\|_F^2\leq \left[\frac{2\varepsilon^2}{(\varepsilon^2-\lambda)}D_{f,\Z^{(0)}}(\X^{(0)},\hat{\X})\right]\left(1-\frac{\nu t}{\|\A\|_F^2}(1-\frac{t \rho\varepsilon^2}{(2\varepsilon^2-2\lambda)})\right)^k\]
 where we assume that the probability of choosing index $i(k)=j$ from $[n_1]$ in Algorithm \ref{alg:LSP} is proportional to $\|\A(j)\|_F^2$ which is a customary choice in Kaczmarz literature \cite{strohmer2009randomized}.
\end{theorem}

\section{Extensions}\label{sec:Ext}
Motivated by accelerating our algorithms for large systems of equations, we present two extensions to our algorithms in this section, i.e., block versions with overlapping and non-overlapping indices.  Overlapping versus non-overlapping indices refers to how the blocks are chosen. In the non-overlapping case, we consider independent blocks where the indices are chosen without replacement. In the overlapping case, we consider blocks of indices chosen with replacement such that the blocks may contain overlapping indices. The block Kaczmarz variant considers a partition of the row indices of $\A$ into a specified number of blocks. Block algorithms with non-overlapping indices are commonly considered in the matrix Kaczmarz setting \cite{needell2014paved,needell2015randomized,necoara2019faster} to improve the convergence of the classical RKA method.

\subsection{Block Kaczmarz with non-overlapping indices}\label{sec:blockconv} We consider Algorithm \ref{alg:LSP} with blocks of indices chosen using a partition of the row indices $[n_1]$. Let the number of blocks be $M$. Then we have a partition $P = \{\tau_1,...,\tau_m\}$ such that the blocks are independent, i.e., the blocks are non-overlapping (NOL). At iteration $k$ of our algorithm we choose one block $\tau_k$ of indices and this choice may be performed randomly or cyclically. Choosing to use a partition with a larger number of blocks in turn accelerates each iteration of the algorithm as fewer row indices are selected at each step.
Using Theorem \ref{thm:randconv}, a linear convergence rate in expectation can be shown for Algorithm \ref{alg:LSP} with non-overlapping blocks.

\begin{theorem}
    Consider the objective function
\begin{align}
    f(\X) = \sum_{i_1=1}^{n_1}\cdots\sum_{i_m=1}^{n_m}\lambda\log\left(1+\frac{|\X_{i_1,...,i_m}|}{\varepsilon}\right)+\frac12\norm{\X}_F^2.
\end{align}
  If $\sqrt{\lambda}<\varepsilon$, and  $0<\frac{\nu t}{ M\|\A\|_F^2}(1-\frac{t \rho \varepsilon^2}{2\varepsilon^2 -2\lambda})<1$, then $\X^{(k)}$ converges linearly to the solution $\hat{\X}$ in expectation
via
\[\mathbb{E}\|\X^{(k)}-\hat{\X}\|_F^2\leq\left[\frac{2\varepsilon^2}{\varepsilon^2 - \lambda}D_{f,\Z^{(0)}}(\X^{(0)},\hat{\X})\right]\left(1-\frac{\nu t}{M\|\A\|_F^2}(1-\frac{t \rho \varepsilon^2}{2\varepsilon^2 -2\lambda})\right)^k.\]
\end{theorem}
\begin{proof}

 The key part of proving this theorem lies in the following bound where $M$ is the number of blocks.
    From \eqref{convthm} we have
    \begin{align*}
    &D_{f,\Z^{(k+1)}}(\X^{(k+1)},\X)\\
    &\leq D_{f,\Z^{(k)}}(\X^{(k)},\X) - t\left(1-\frac{t \rho \varepsilon^2}{2\varepsilon^2 -2\lambda}\right)\frac{\| A(\tau_k)*_L(\X^{(k)}-\X)\|_F^2}{\|\A(\tau_k)\|_F^2}
\end{align*}
and we want to analyze the expectation of $\frac{\| \A(\tau_k)*_L(\X^{(k)}-\X)\|_F^2}{\|\A(\tau_k)\|_F^2}$, where $\A(\tau_k) = \A(\tau_k,:,:,...,:)$. Then let $\mathbb{E}_c$ be the expectation conditioned on $\tau_0,...\tau_{k-1}$.

By the definition of expectation and Lemma~\ref{lemma:nu}, we have
\begin{align*}
   \mathbb{E}_c\left[\frac{\| \A(\tau_k)*_L(\X^{(k)}-\X)\|_F^2}{\|\A(\tau_k)\|_F^2}\right] &= \frac{1}{M}\sum_{j\in \tau_k}\frac{\| \A(j)*_L(\X^{(k)}-\X)\|_F^2}{\|\A(j)\|_F^2} \\
   & = \frac{1}{M}\frac{\| \A*_L(\X^{(k)}-\X)\|_F^2}{\|\A\|_F^2}\\
   & \geq  \frac{1}{M} \frac{\nu}{\|\A\|_F^2}D_{f,\Z^{(k)}}(\X^{(k)},\hat{\X}).
\end{align*} The last inequality follows from \eqref{nulemma}. Then taking the expectation leads to
\begin{align*}
    &\mathbb{E}\left[D_{f,\Z^{(k+1)}}(\X^{(k+1)},\hat{\X})\right]\\
    & \leq
    \mathbb{E}\left[D_{f,\Z^{(k)}}(\X^{(k)},\hat{\X})-t\left(1-\frac{t \rho \varepsilon^2}{2\varepsilon^2 -2\lambda}\right)\frac{\| \A(\tau_k)*_L(\X^{(k)}-\X)\|_F^2}{\|\A(\tau_k)\|_F^2}\right]\\
    & = \mathbb{E}_{\tau_0,...,\tau_{k-1}}\mathbb{E}_c\left[D_{f,\Z^{(k)}}(\X^{(k)},\hat{\X})-t\left(1-\frac{t \rho \varepsilon^2}{2\varepsilon^2 -2\lambda}\right)\frac{\| \A(\tau_k)*_L(\X^{(k)}-\X)\|_F^2}{\|\A(\tau_k)\|_F^2}\right]\\
    &\leq \mathbb{E}_{\tau_0,...,\tau_{k-1}}\left[D_{f,\Z^{(k)}}(\X^{(k)},\hat{\X})-t\left(1-\frac{t \rho \varepsilon^2}{2\varepsilon^2 -2\lambda}\right)\frac{1}{M} \frac{\nu}{\|\A\|_F^2}D_{f,\Z^{(k)}}(\X^{(k)},\hat{\X})\right]\\
    &\leq \left(1-\frac{\nu t}{M\|\A\|_F^2}\left(1-\frac{t \rho \varepsilon^2}{2\varepsilon^2 -2\lambda}\right)\right)\mathbb{E}\left[D_{f,\Z^{(k)}}(\X^{(k)},\hat{\X})\right].
\end{align*}
After applying the above inequality repeatedly, we can bound the expected Bregman distance after $k$ iterations in terms of the initial Bregman distance.
\[
\mathbb{E}\left[D_{f,\Z^{(k)}}(\X^{(k)},\hat{\X})\right]\leq \left(1-\frac{\nu t}{M\|\A\|_F^2}\left(1-\frac{t \rho \varepsilon^2}{2\varepsilon^2 -2\lambda}\right)\right)^k\left[D_{f,\Z^{(0)}}(\X^{(0)},\hat{\X})\right].
\]
Then by \eqref{bregmanprop}, we have
\begin{align*}
    \mathbb{E}\|\X^{(k)}-\hat{\X}\|_F^2&\leq \frac{2\varepsilon^2}{\varepsilon^2 - \lambda}D_{f,\Z^{(0)}}\mathbb{E}\left[D_{f,\Z^{(k)}}(\X^{(k)},\hat{\X})\right]\\
    &\leq \left[\frac{2\varepsilon^2}{\varepsilon^2 - \lambda}D_{f,\Z^{(0)}}(\X^{(0)},\hat{\X})\right]\left(1-\frac{\nu t}{M\|\A\|_F^2}\left(1-\frac{t \rho \varepsilon^2}{2\varepsilon^2 -2\lambda}\right)\right)^k.
\end{align*}
Note that this result relies on $M$, the number of blocks.

\end{proof}

\subsection{Block Kaczmarz with overlapping indices} Another variant of the regularized Kaczmarz algorithm is the block Kaczmarz with overlapping indices (OL). Specifically, at each iteration of the algorithm, a block of indices is selected from $[n_1]$ with replacement, i.e., the blocks may overlap. This method is utilized in the numerical experiments of \cite{chen2021regularized}, where the size of each block is uniform. Under certain conditions and up to some reordering, the two block algorithm cases can be equivalent. For example, this is true when the block size is uniform, $n_1$ is divisible by the block size, and the algorithm is performed cyclically.

\section{Numerical Experiments}\label{sec:num}

In this section, we demonstrate the effectiveness of our proposed algorithms on synthetic and real-world data for sparse and low-rank high-order tensor recovery applications. Throughout the numerical experiments, we utilize the following quantitative metrics to evaluate the recovery performance of various tensor recovery algorithms:
\begin{enumerate}
\item The relative error (RE) at each iteration is defined as
 \begin{align}
    RE(\X^{(k)},\hat{\X})= \frac{\|\hat{\X}-\X^{(k)}\|_F}{\|\hat{\X}\|_F},\label{relative error}
 \end{align}
 where $\X^{(k)}$ is the estimation of the ground truth $\hat{\X}$ at the $k$-th iteration.

 \item The Peak Signal-to-Noise Ratio (PSNR)\cite{qin2022low}  is calculated as
\[\text{PSNR} = 10 \log_{10}(n_1\times n_2\times \cdots \times n_m\|\X\|_\infty^2/\|\hat{\X}-\X\|_F^2),\]
where $\X$ is the original tensor of size $n_1\times n_2\times \ldots\times n_m$ and $\hat{\X}$ is the approximated tensor. This definition is a generalized version of the standard PSNR in the image processing community.

\item The Structural Similarity Index Measure (SSIM) \cite{wang2004image} calculates feature similarity by combining similarity measures for luminance, contrast, and structure.

\item The Feature Similarity Index Measure (FSIM) \cite{zhang2011fsim} utilizes phase congruency and gradient magnitude to characterize local image quality.
\end{enumerate}

For the synthetic experiments, we adopt the relative error for performance assessment. But for the experiments on real-world data, we adopt all four qualitative metrics. All the experiments were implemented in MATLAB 2023b on a desktop computer with Intel CPU i7-1065G7 CPU RAM 12GB with Windows 11. The demo codes of our proposed algorithm are available\footnote{https://github.com/khenneberger}.

\subsection{Synthetic experiments}
In the following numerical experiments we apply Algorithm \ref{alg:LSP}, LSPK-S,  and Algorithm \ref{alg:LSPsing}, LSPK-L, to a synthetic data set where the tensors $\A\in\R^{10\times 2\times 10\times 10}$ and ground truth $\X\in\R^{2\times 10\times 10\times 10}$ are randomly generated. For the sparse recovery algorithm LSPK-S, $20\%$ of the entries of $\X$ are randomly set to zero.

We conduct various comparisons using this synthetic dataset. First, we evaluate the convergence of our two algorithms, LSPK-S and LSPK-L, with block OL and block NOL. Next, we analyze the convergence of our algorithms under three distinct linear transforms. Finally, we assess the convergence of our algorithms with different block sizes for the block OL algorithms and varying numbers of blocks for the block NOL algorithms.

When applying the block OL algorithms, a different index block $i(k)$ is chosen at each iteration $k$. The block size, $\beta$, is constant at each iteration and our choice of indexing can be cyclic or randomized. When applying the block NOL algorithms, the rows indices of $\A$ are partitioned into $d$ blocks and at iteration $k$, block $\tau_k$ is chosen either cyclically or randomly. Since the block size and the number of blocks effect convergence, one may tune these parameters to find the optimal choice. Throughout this section we set the parameters as $\beta= 7, \lambda = 0.001$ and $t=1$, $\varepsilon = 0.1$, $d = 3$ and $L=$ \verb 'fft' unless specified otherwise. Fig.~\ref{fig:batchblock} plots the relative error curves averaged over 50 trials. In both the LSPK-S and LSPK-L algorithms, the random block NOL variant demonstrates rapid convergence, while the random block OL variant shows the slowest convergence.

\begin{figure}
    \centering
     \begin{tabular}{c c}
     \includegraphics[width = .47\textwidth]{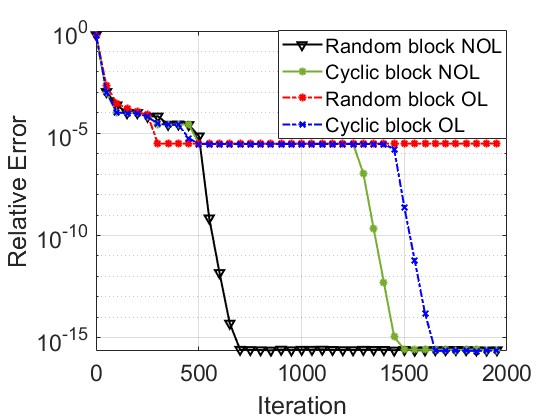}
     & \includegraphics[width = .47\textwidth]{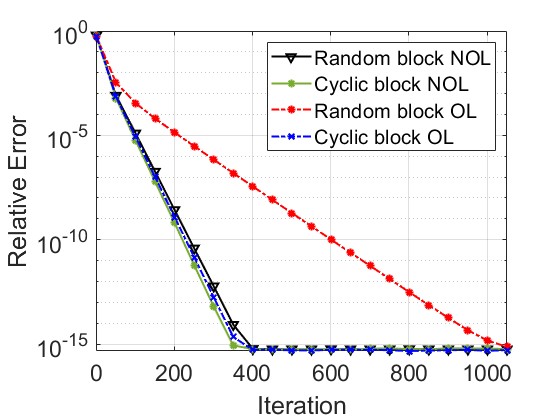} \\
        LSPK-S (Alg.~\ref{alg:LSP}) &  LSPK-L (Alg.~\ref{alg:LSPsing})
     \end{tabular}
     \caption{Convergence for cyclic and random blocking with overlapping and non-overlapping indices}
     \label{fig:batchblock}
 \end{figure}

Next, we compare the convergence of  Algorithms \ref{alg:LSP} and \ref{alg:LSPsing}  with different linear transforms: the Fast Fourier Transform (FFT), the Discrete Cosine Transform (DCT), and the Discrete Wavelet Transform (DWT) with the Daubechies 5 wavelet `\verb"db5"'. For this experiment, we use random blocks of size $\beta = 7$. Figure \ref{fig:lintrans} plots the relative error averaged over 50 trials produced by using the linear transforms FFT, DCT, and DWT. For this synthetic data, the regularized Kaczmarz algorithm with the FFT-based t-product performs better than the versions using the DCT or DWT.

 \begin{figure}
     \centering
     \begin{tabular}{c c}
        \includegraphics[width = .47\textwidth]{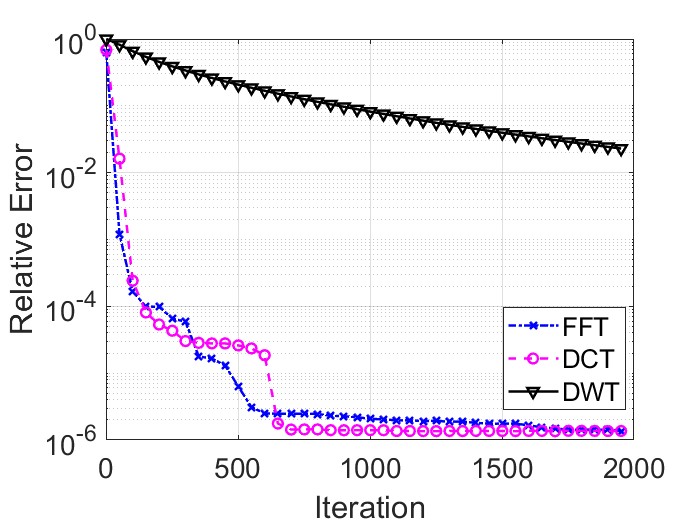}
        &  \includegraphics[width = .47\textwidth]{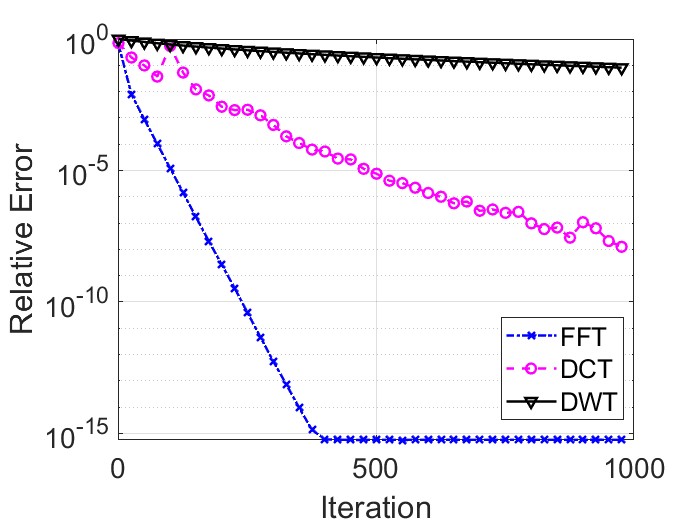} \\
       LSPK-S (Alg.~ \ref{alg:LSP}) &  LSPK-L (Alg.~ \ref{alg:LSPsing})
     \end{tabular}
     \caption{Convergence under different linear transforms}
     \label{fig:lintrans}
     %\vspace{-10pt}
 \end{figure}
We also compare the convergence of our algorithms with different block sizes. Figure \ref{fig:batchblocksize} plots the relative error curves for our two algorithms where blocking with and without overlapping are both performed randomly. Figure \ref{fig:timecomp} presents the time in seconds per iteration for varying block sizes (the overlapping case) and number of blocks (the non-overlapping case). The fewer number of blocks, the more rows are used in each iteration of the algorithm and thus the more time used. Similarly, the larger the block size, the more rows are used in each iteration of the algorithm and thus more time is needed to run 2000 iterations. From these observations, we see the trade-off of time and convergence when selecting block sizes and the number of blocks. For example, when the block size is one, the algorithm is faster for each iteration but slower to converge when compared to other block sizes.
 \begin{figure}
     \centering
     \begin{tabular}{c c}
        \includegraphics[width = .47\textwidth]{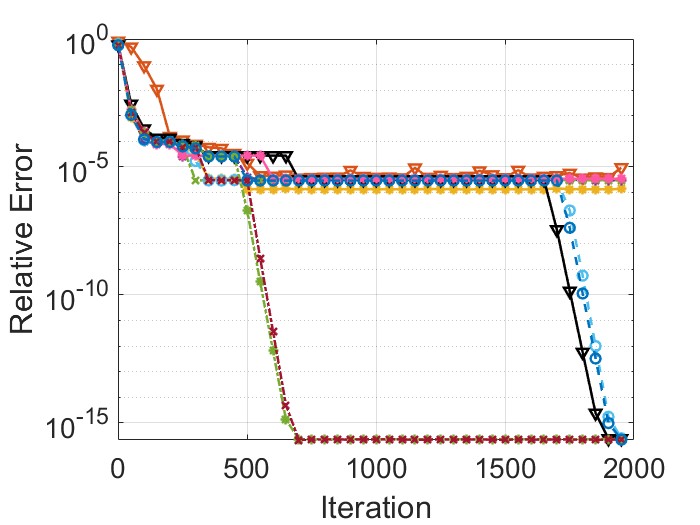}
        &  \includegraphics[width = .47\textwidth]{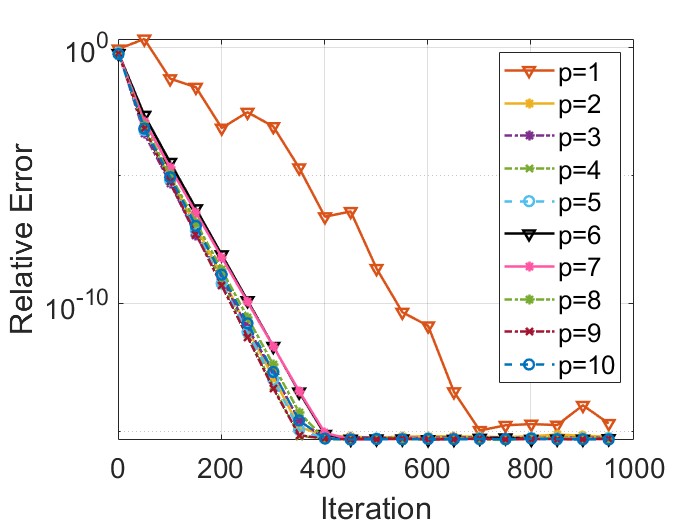} \\
       LSPK-S (Alg.~ \ref{alg:LSP}) OL &  LSPK-L (Alg.~ \ref{alg:LSPsing}) OL\\
       \includegraphics[width = .47\textwidth]{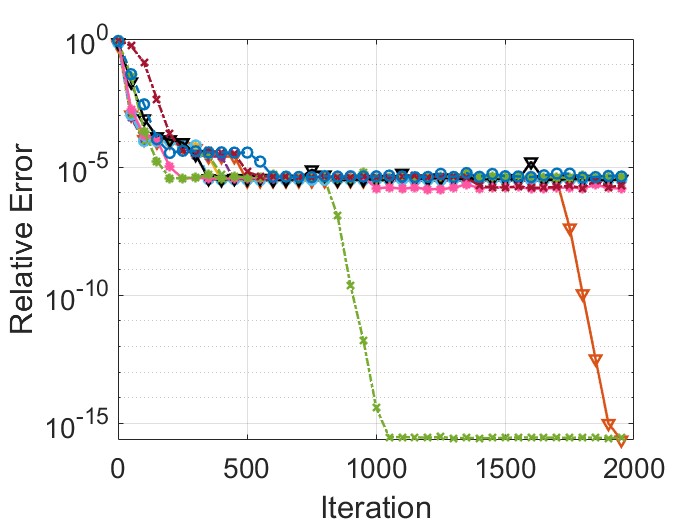}
       &  \includegraphics[width = .47\textwidth]{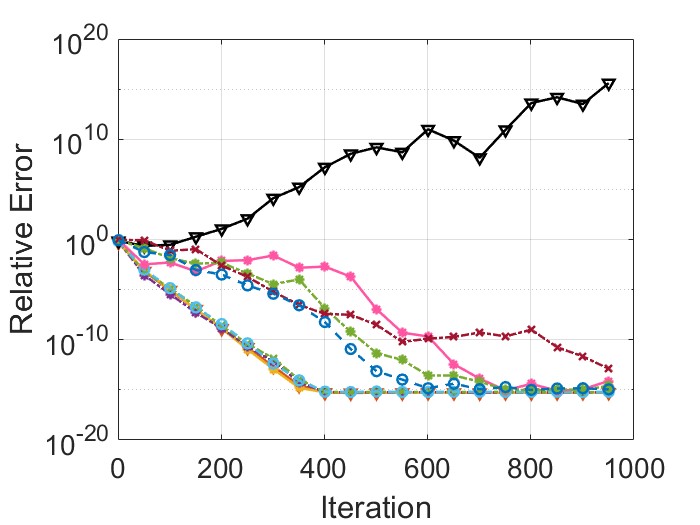} \\
       LSPK-S (Alg.~ \ref{alg:LSP}) NOL &  LSPK-L (Alg.~ \ref{alg:LSPsing}) NOL
     \end{tabular}
     \caption{Convergence with different number of blocks and block sizes. ``NOL'' refers to blocks with non-overlapping indices where ``p'' indicates the number of blocks. ``OL'' refers to blocks with overlapping indices where ``p'' indicates the block sizes. }
     \label{fig:batchblocksize}
 \end{figure}
 \begin{figure}
     \centering
     \includegraphics[width = .47\textwidth]{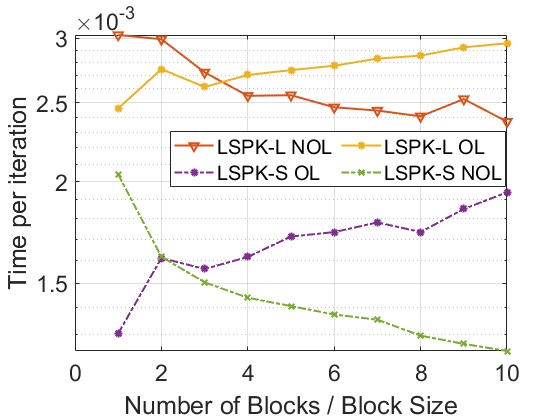}
     \caption{Comparison of running time in seconds per iteration for varying number of blocks and block sizes. ``NOL'' refers to blocks with non-overlapping indices where the horizontal axis indicates the number of blocks. ``OL'' refers to blocks with overlapping indices where the horizontal axis indicates the block sizes.}
     \label{fig:timecomp}
 \end{figure}

\subsection{Tensor Destriping}

In this experiment, we consider a low-rank image destriping problem. Striping refers to the presence of straight horizontal or vertical stripes caused by line-scanning, and image destriping aims to remove stripes or streaks from the visual data. We consider Extended YaleFace Dataset B \cite{georghiades2001few}, a fourth order dataset of 38 subjects photographed under different illumination conditions. We consider a subset of the original data such that our tensor is of size $48\times 42\times 64\times 38$. This includes $38$ subjects under $64$ different illuminations where each image is $48\times 42$ pixels.

For this application, we compare our approach with two other related recent high-order tensor recovery methods: HTNN-FFT \cite{qin2022low} and t-CTV \cite{wang2023guaranteed}. Based on the assumption that the desired tensor $\hat{\X}$ possesses a low t-SVD rank, HTNN-FFT solves the following model
\begin{align}
\min_{\X}\|\X\|_{\star,L}\quad \text{s.t.} \quad P_\Omega(\X) = P_\Omega(\hat{\X}),
\end{align}
where $P_\Omega(\cdot)$ represents the projection onto the observed index set $\Omega$, and $\|\cdot\|_{\star,L}$ stands for the high-order tensor nuclear norm. More recently, the t-CTV algorithm \cite{wang2023guaranteed} assumes the low-rankness of gradient tensors of the underlying tensor. Consequently, it solves the following model
\begin{align}
\min_{\X}\|\X\|_{\text{t-CTV}}\quad \text{s.t.} \quad P_\Omega(\X) = P_\Omega(\hat{\X}),
\end{align}
where $\|\cdot\|_{\text{t-CTV}}$ is the tensor correlated total variation norm. Note that both HTNN-FFT and t-CTV algorithms are derived by applying the alternating direction method of multipliers (ADMM).

To start with, the tensor $\A$ is chosen to be facewise diagonal with all diagonal entries equal to one except in the rows which will establish the striping effect. We choose $0.01$ for the diagonal entries where we wish to create a stripe. After parameter tuning, we choose the optimal parameters $\lambda = 0.1, t = 1,$ and $ \beta = 1$, and $\varepsilon = 1$. Figure \ref{fig:destripe} shows results on the three selected images demonstrating both the quantitative and qualitative success of Algorithm \ref{alg:LSPsing} as compared with two state-of-the-art methods for tensor completion: HTNN-FFT \cite{qin2022low} and t-CTV \cite{wang2023guaranteed}. Figure \ref{fig:destripeErr} plots the
relative error of our proposed method and the two comparison methods over 500 iterations. The proposed algorithm demonstrates superior performance compared to both the comparison methods. Additionally, HTNN-FFT exhibits slightly better performance than t-CTV in this application.

\begin{figure}[ht]
    \centering
    \setlength{\tabcolsep}{2pt}
    \begin{tabular}{ccccc}
        Original & Observed &  t-CTV& HTNN-FFT & Proposed \\
        {\includegraphics[width = 0.19\textwidth]{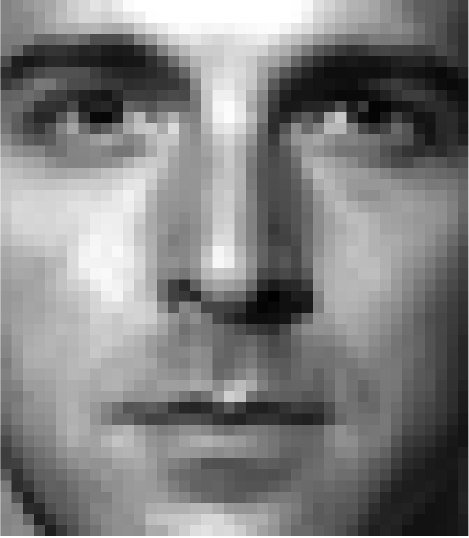}}& {\includegraphics[width = 0.19\textwidth]{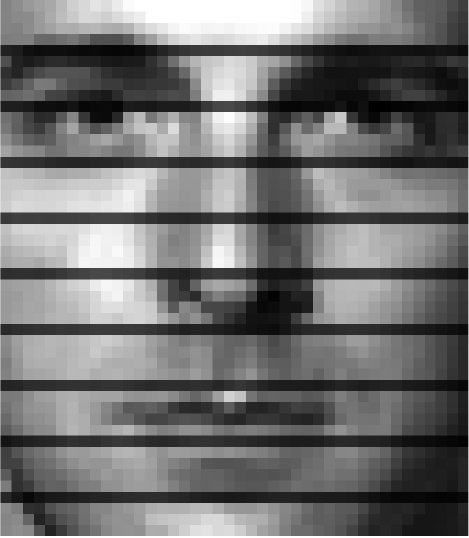}}&{\includegraphics[width = 0.19\textwidth]{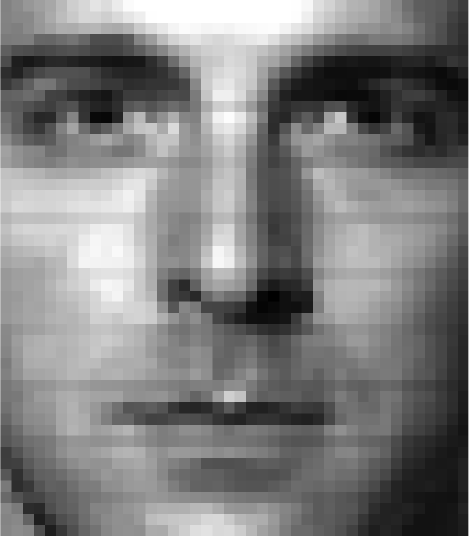}}& {\includegraphics[width = 0.19\textwidth]{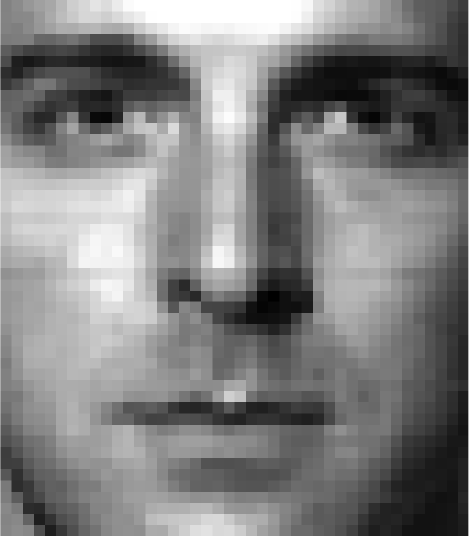}}& {\includegraphics[width = 0.19\textwidth]{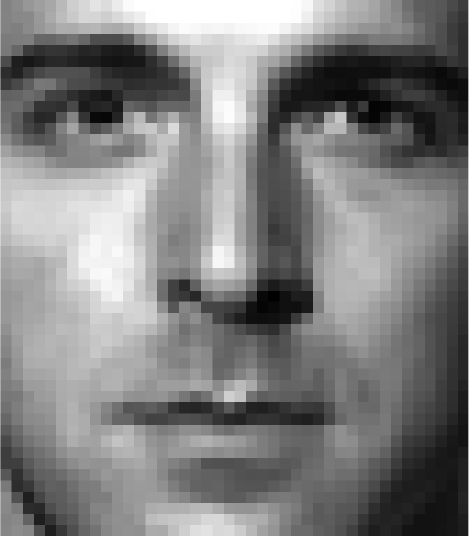}} \\
      {\includegraphics[width = 0.19\textwidth]{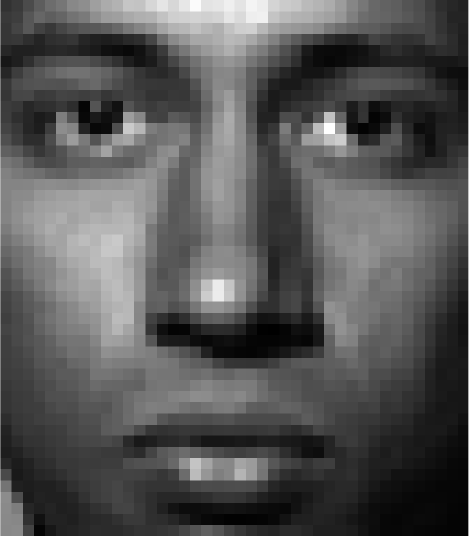}}& {\includegraphics[width = 0.19\textwidth]{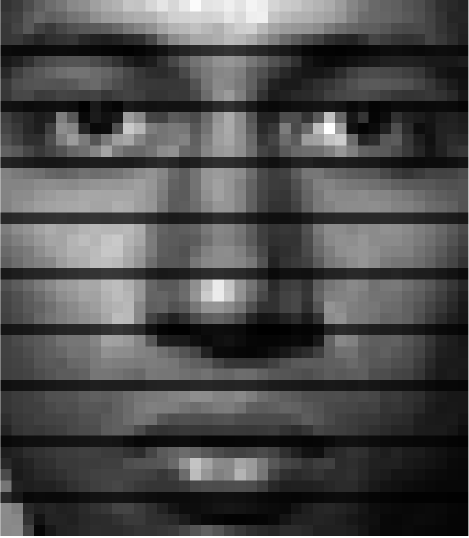}}&  {\includegraphics[width = 0.19\textwidth]{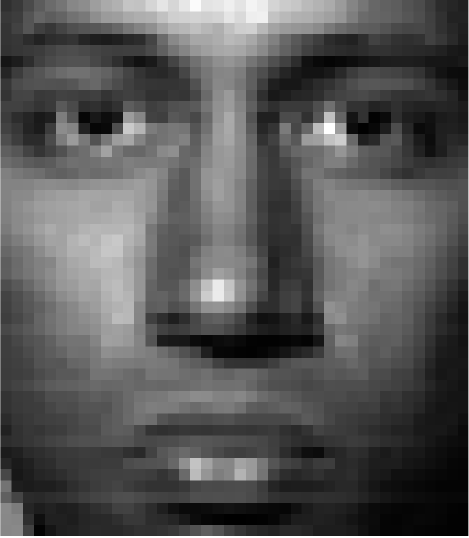}}& {\includegraphics[width = 0.19\textwidth]{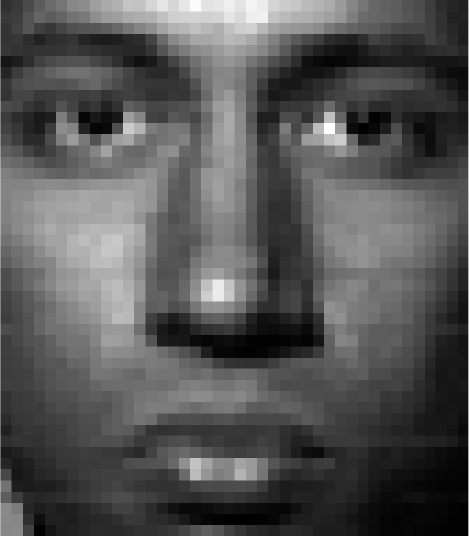}}&{\includegraphics[width = 0.19\textwidth]{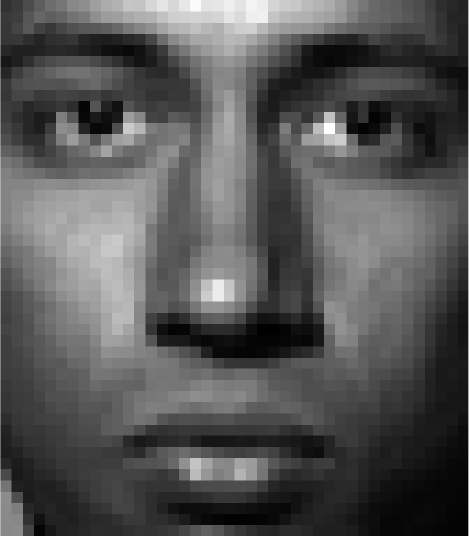}}\\
     {\includegraphics[width = 0.19\textwidth]{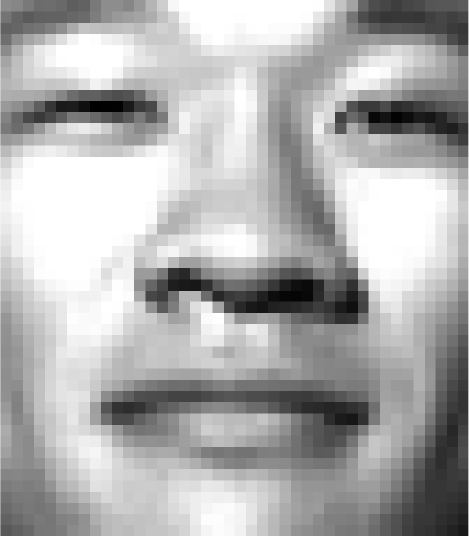}}& {\includegraphics[width = 0.19\textwidth]{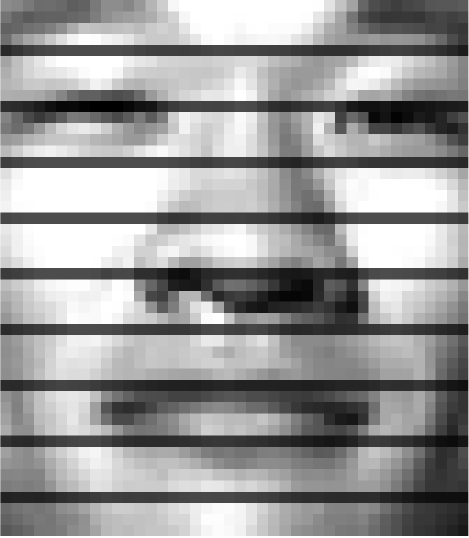}}& {\includegraphics[width = 0.19\textwidth]{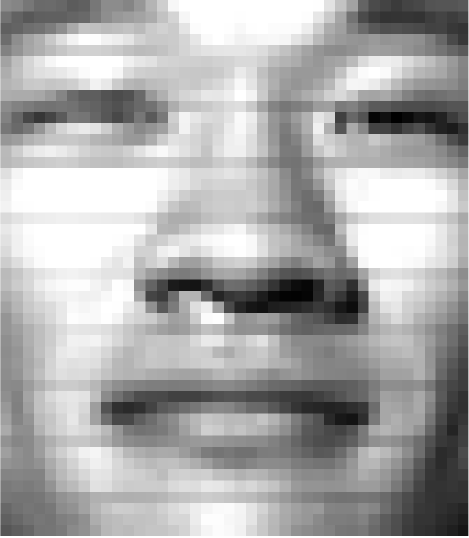}}& {\includegraphics[width = 0.19\textwidth]{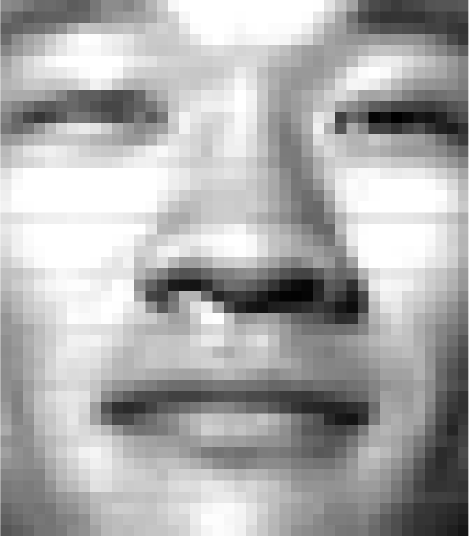}}&{\includegraphics[width = 0.19\textwidth]{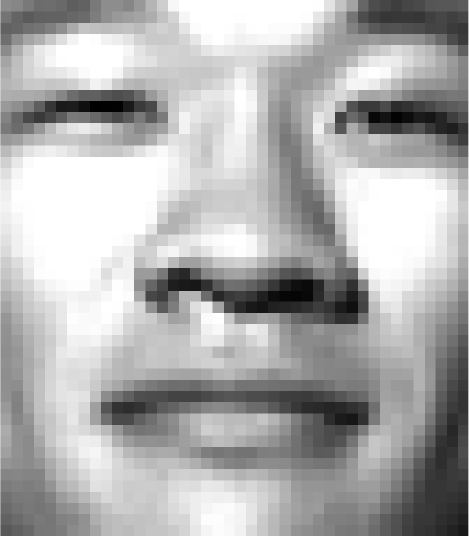}}\\
     & & RE $=0.07591$ & RE $= 0.0542$ & RE $= 1.06\times 10^{-15}$
        \end{tabular}
        %\vspace{-5pt}
        \caption{Visual results where every 5th row is a stripe}
        \label{fig:destripe}
\end{figure}
\begin{figure}
    \centering    \includegraphics[width = 0.5\textwidth]{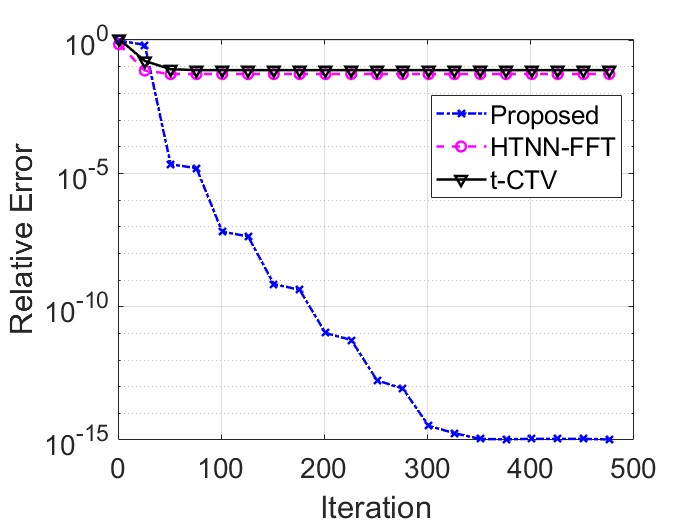}
    \caption{Relative error results where every 5th row is a stripe}   \label{fig:destripeErr}
\end{figure}

We further consider a random sampling of rows with varying sampling rates (SR). Table \ref{tab:comp4} compares our method with HTNN-FFT and t-CTV for several different sampling rates over the rows of the dataset. These results are the average of 50 trials where sampling is performed randomly over the rows of our data for each trial and the average time per second (/s) for each iteration is recorded in Table \ref{tab:comp4}.  In this application, the proposed algorithm outperforms both comparison methods in terms of quantitative metrics. As observed in Table \ref{tab:comp4} and Figure \ref{fig:destripeErr}, the relative error of the proposed algorithm reaches machine epsilon within a small number of iterations, demonstrating its rapid convergence. Despite t-CTV running much slower than HTNN-FFT, it exhibits significantly better performance across all quantitative metrics.
\begin{table}[ht]
    \centering
    \begin{tabular}{|c|c|c|c|c|c|c|}
\hline
     \textbf{SR}&\textbf{Method} & \textbf{PSNR}  & \textbf{SSIM}  & \textbf{FSIM} & \textbf{RE} & \textbf{time/s} \\
     \hline\hline
     \multirow{3}{*}{80\%}  & Proposed method & $228.276$   & $1.000$  & $1.000$ & $1.05409\times 10^{-11}$ &  $1.1905$\\
     \cline{2-7}& HTNN-FFT  & $15.5723$ &$0.9986$  &  $0.9713$   &  $0.4553$& $0.8477$\\

     \cline{2-7}& t-CTV     &     $28.7759$  &  $0.9999$ &   $0.9965$  & $0.10084$  &$4.2716$\\
     \hline\hline
     \multirow{3}{*}{70\%} & Proposed method & $228.276$   & $1.000$  & $1.000$ & $1.05409\times 10^{-11}$ &  $1.0879$\\
     \cline{2-7}& HTNN-FFT & $14.0413$   & $0.9975$  &  $0.9627$ & $0.5430$ & $ 0.7962$\\

     \cline{2-7}& t-CTV &  $25.3434$    & $0.9998$  &  $0.9925$  & $0.1525$ &  $4.2348$\\
     \hline\hline
     \multirow{3}{*}{60\%} &Proposed method & $228.276$   & $1.000$  & $1.000$ & $1.05409\times 10^{-11}$ &  $1.0773$\\
     \cline{2-7}& HTNN-FFT & $12.7158$    &  $0.9960$ &   $0.9538$  & $0.6324$&    $0.7436$\\

     \cline{2-7}& t-CTV &  $21.8713$     & $0.9994$   & $0.9856$& $0.2319$ & $4.1702$\\
     \hline\hline
     \multirow{3}{*}{50\%} & Proposed method & $228.276$   & $1.000$  & $1.000$ & $1.05409\times 10^{-11}$ &  $1.0582$\\
     \cline{2-7} & HTNN-FFT & $11.7498$   &  $0.9947$ &   $0.9467$& $0.7068$ & $0.6974$\\
     \cline{2-7}& t-CTV &  $19.5387$   &  $0.9990$  &  $0.9793$&$0.3060$ & $4.0730$\\
   \hline
\end{tabular}
    \caption{Performance Comparison for $4$th order tensor recovery with random row sampling}
    \label{tab:comp4}
\end{table}

\section{Conclusion}\label{sec:Con}

In this work, we propose novel log-sum penalty regularizied models for recovering tensors with either sparse or low-rank structures, utilizing a generalized t-product for high order tensors based on some invertible transform. To solve them, we develop regularized Kaczmarz algorithms which alternates the Kaczmarz step and proximity operator. Convergence guarantees are discussed for the proposed algorithms. In addition, we extend the proposed algorithms with two block variants for handling large data sets to further boost the computational efficiency. A variety of numerical experiments on synthetic and real data sets demonstrate the outstanding performance of our algorithms over the other state-of-the-art methods. In our future work, we will adapt this tensor regularization framework to other tensor decomposition settings and develop their accelerated versions with gradient variance stabilization techniques for various big-data applications.

\section*{Acknowledgements}
The research of Qin is supported partially by the NSF grant DMS-1941197. The research of Henneberger is supported partially by the NSF grant DMS-1941197 and DMS-2110731. In addition, the authors would like to thank Dr. Ding Lu for the helpful discussion.

\bibliographystyle{unsrt}
\bibliography{references}

\begin{thebibliography}{10}

\bibitem{karczmarz1937angenaherte}
Stefan Karczmarz.
\newblock Angenaherte auflosung von systemen linearer glei-chungen.
\newblock {\em Bull. Int. Acad. Pol. Sic. Let., Cl. Sci. Math. Nat.}, pages
  355--357, 1937.

\bibitem{gordon1970algebraic}
Richard Gordon, Robert Bender, and Gabor~T Herman.
\newblock Algebraic reconstruction techniques (art) for three-dimensional
  electron microscopy and x-ray photography.
\newblock {\em Journal of theoretical Biology}, 29(3):471--481, 1970.

\bibitem{zhou2013tensor}
Hua Zhou, Lexin Li, and Hongtu Zhu.
\newblock Tensor regression with applications in neuroimaging data analysis.
\newblock {\em Journal of the American Statistical Association},
  108(502):540--552, 2013.

\bibitem{chen2021regularized}
Xuemei Chen and Jing Qin.
\newblock Regularized kaczmarz algorithms for tensor recovery.
\newblock {\em SIAM Journal on Imaging Sciences}, 14(4):1439--1471, 2021.

\bibitem{strohmer2009randomized}
Thomas Strohmer and Roman Vershynin.
\newblock A randomized kaczmarz algorithm with exponential convergence.
\newblock {\em Journal of Fourier Analysis and Applications}, 15(2):262, 2009.

\bibitem{needell2014stochastic}
Deanna Needell, Rachel Ward, and Nati Srebro.
\newblock Stochastic gradient descent, weighted sampling, and the randomized
  kaczmarz algorithm.
\newblock {\em Advances in neural information processing systems}, 27, 2014.

\bibitem{wright2015coordinate}
Stephen~J Wright.
\newblock Coordinate descent algorithms.
\newblock {\em Mathematical programming}, 151(1):3--34, 2015.

\bibitem{needell2014paved}
Deanna Needell and Joel~A Tropp.
\newblock Paved with good intentions: analysis of a randomized block kaczmarz
  method.
\newblock {\em Linear Algebra and its Applications}, 441:199--221, 2014.

\bibitem{needell2015randomized}
Deanna Needell, Ran Zhao, and Anastasios Zouzias.
\newblock Randomized block kaczmarz method with projection for solving least
  squares.
\newblock {\em Linear Algebra and its Applications}, 484:322--343, 2015.

\bibitem{marshall2023optimal}
Nicholas~F Marshall and Oscar Mickelin.
\newblock An optimal scheduled learning rate for a randomized kaczmarz
  algorithm.
\newblock {\em SIAM Journal on Matrix Analysis and Applications},
  44(1):312--330, 2023.

\bibitem{wang2023solving}
Xuezhong Wang, Maolin Che, Changxin Mo, and Yimin Wei.
\newblock Solving the system of nonsingular tensor equations via randomized
  kaczmarz-like method.
\newblock {\em Journal of Computational and Applied Mathematics}, 421:114856,
  2023.

\bibitem{ma2022randomized}
Anna Ma and Denali Molitor.
\newblock Randomized kaczmarz for tensor linear systems.
\newblock {\em BIT Numerical Mathematics}, 62(1):171--194, 2022.

\bibitem{prater2022proximity}
Ashley Prater-Bennette, Lixin Shen, and Erin~E Tripp.
\newblock The proximity operator of the log-sum penalty.
\newblock {\em Journal of Scientific Computing}, 93(3):67, 2022.

\bibitem{zhou2023iterative}
Xin Zhou, Xiaowen Liu, Gong Zhang, Luliang Jia, Xu~Wang, and Zhiyuan Zhao.
\newblock An iterative threshold algorithm of log-sum regularization for sparse
  problem.
\newblock {\em IEEE Transactions on Circuits and Systems for Video Technology},
  2023.

\bibitem{chen2021logarithmic}
Lin Chen, Xue Jiang, Xingzhao Liu, and Zhixin Zhou.
\newblock Logarithmic norm regularized low-rank factorization for matrix and
  tensor completion.
\newblock {\em IEEE Transactions on Image Processing}, 30:3434--3449, 2021.

\bibitem{qin2022low}
Wenjin Qin, Hailin Wang, Feng Zhang, Jianjun Wang, Xin Luo, and Tingwen Huang.
\newblock Low-rank high-order tensor completion with applications in visual
  data.
\newblock {\em IEEE Transactions on Image Processing}, 31:2433--2448, 2022.

\bibitem{lu2018tensor}
Canyi Lu.
\newblock Tensor-tensor product toolbox.
\newblock {\em arXiv preprint arXiv:1806.07247}, 2018.

\bibitem{gong2013general}
Pinghua Gong, Changshui Zhang, Zhaosong Lu, Jianhua Huang, and Jieping Ye.
\newblock A general iterative shrinkage and thresholding algorithm for
  non-convex regularized optimization problems.
\newblock In {\em international conference on machine learning}, pages 37--45.
  PMLR, 2013.

\bibitem{rockafellar2015convex}
Ralph~Tyrell Rockafellar.
\newblock {\em Convex Analysis}.
\newblock Princeton university press, 2015.

\bibitem{necoara2019faster}
Ion Necoara.
\newblock Faster randomized block kaczmarz algorithms.
\newblock {\em SIAM Journal on Matrix Analysis and Applications},
  40(4):1425--1452, 2019.

\bibitem{wang2004image}
Zhou Wang, Alan~C Bovik, Hamid~R Sheikh, and Eero~P Simoncelli.
\newblock Image quality assessment: from error visibility to structural
  similarity.
\newblock {\em IEEE transactions on image processing}, 13(4):600--612, 2004.

\bibitem{zhang2011fsim}
Lin Zhang, Lei Zhang, Xuanqin Mou, and David Zhang.
\newblock Fsim: A feature similarity index for image quality assessment.
\newblock {\em IEEE transactions on Image Processing}, 20(8):2378--2386, 2011.

\bibitem{georghiades2001few}
Athinodoros~S. Georghiades, Peter~N. Belhumeur, and David~J. Kriegman.
\newblock From few to many: Illumination cone models for face recognition under
  variable lighting and pose.
\newblock {\em IEEE transactions on pattern analysis and machine intelligence},
  23(6):643--660, 2001.

\bibitem{wang2023guaranteed}
Hailin Wang, Jiangjun Peng, Wenjin Qin, Jianjun Wang, and Deyu Meng.
\newblock Guaranteed tensor recovery fused low-rankness and smoothness.
\newblock {\em IEEE Transactions on Pattern Analysis and Machine Intelligence},
  2023.

\end{thebibliography}
\end{document}